\RequirePackage{fix-cm}
\documentclass[11pt,a4paper,reqno]{article}
\usepackage{graphicx}
\usepackage{mathptmx}   
\usepackage{verbatim}

\usepackage[latin1]{inputenc}  
\usepackage{amsfonts}       
\usepackage{amsmath}
\usepackage{amssymb}
\usepackage{multirow}
\usepackage{graphicx}  
\usepackage{enumerate}
\usepackage{bibentry}
\usepackage{enumerate}
\usepackage{color}
\usepackage{amsmath}

\numberwithin{equation}{section}

\newcommand{\I}{\mathbb{I}}

\newcommand{\eqdist}{\stackrel{d}{=}}

\usepackage{amsthm}
\newtheorem{theorem}{Theorem}
\newtheorem{corollary}{Corollary}
\newtheorem{proposition}{Proposition}
\newtheorem{lemma}{Lemma}
\newtheorem*{mainprob}{The Main Problem}

\theoremstyle{definition}
\newtheorem{example}{Example}
\newtheorem{definition}{Definition}
\newtheorem{remark}{Remark}

\author{Jaakko Lehtomaa}

\title{On asymptotic scales of independently stopped random sums}

\begin{document}

\maketitle

\begin{abstract}
We study randomly stopped sums via their asymptotic scales. First, finiteness of moments is considered. To generalise this study, asymptotic scales applicable to the class of all heavy-tailed random variables are used. The stopping is assumed to be independent of the underlying process, which is a random walk. 

The main result enables one to identify whether the asymptotic behaviour of a stopped sum is dominated by the increment, or the stopping variable. As a consequence of this result, new sufficient conditions for the moment determinacy of compounded sums are obtained.
\end{abstract}


\section{Introduction}

Let $(\Omega, \mathcal{F},P)$ be a probability space where all subsequent random variables are defined. Suppose $N$ is a random variable taking values in $\mathbb{N}=\{1,2,3\ldots\}$. The variable
\begin{equation}\label{perus}
S_N=\sum_{k=1}^N X_k
\end{equation}
is called a \emph{randomly stopped sum}. Here $(X_i)$ is a sequence of real valued random variables called \emph{increments}. If $(X_i)$ is an IID (independent and identically distributed) sequence, then the variable of Formula \eqref{perus} is called a \emph{stopped random walk}.

Randomly stopped sums are one of the most studied cases of randomly stopped processes, see e.g. \cite{gut2009stopped,gut1986converse}. Typically, one is interested to know how the tail of the variable $S_N$ is affected by tails of its increments and stopping variables. Identification of the dominant variable of the stopped sum has applications in e.g. insurance, where the compounded variable $S_N$ is used to model the aggregate loss of a company. In this setting, the aim is to find out if large losses are mainly caused by few big increments or unusually large amounts of small increments.

\begin{mainprob} When does the increment $X$ or the stopping $N$ alone dominate the asymptotic behaviour of the stopped random walk $S_N$?
\end{mainprob} 

Asymptotic form of the tail of $S_N$ has been studied extensively under different distributional assumptions. Specifically, the central topic has been to discover whether the tail behaviour of $P(S_N>x)$, as $x \to \infty$,  corresponds to that of 
\begin{enumerate}[I]
\item $E(N)P(X_1>x)$ \label{A} or
\item $P(E(X_1)N>x)$. \label{B}
\end{enumerate}
In \ref{A}, the increment $X$ is dominating while in \ref{B} the dominating variable is $N$. The case \ref{A} has been studied in \cite{schmidli1999compound, foss2007lower, denisov2008lower, denisov2008lower2}. The case \ref{B} seems to be less studied than case \ref{A}. Although, references for the latter can be found from \cite{denisov2010asymptotics, robert2008tails}. 

The main contribution of the present paper is to offer a simple method that allows one to identify, on a rough scale, which of the situations \ref{A} or \ref{B} is present. We begin with results about moments and expand the study to results concerning asymptotic scales. This leads to an intuitively satisfying result in Theorem \ref{theorem1}, where heavier increment $X$ leads to \ref{A} while heavier stopping $N$ leads to \ref{B}.

The result does not require detailed assumptions about the distribution classes of random variables. Furthermore, using rough scales it becomes possible to analyse situations where the dominant variable is only slightly heavier than the other variable. Recall that a random variable is called \emph{heavy-tailed} if no exponential moments exist and \emph{light-tailed} otherwise. It will be shown that the main result holds when one of the variables $X$ and $N$ is heavy- and the other light-tailed without any further assumptions. 

The second result of the paper considers the moment determinacy of random sums. It turns out that the asymptotic scale of the random sum is closely related to the moment determinacy of $S_N$. A random variable $X\geq 0$ with finite moments $E(X^k)$ of all orders $k\in \mathbb{N}$ is \emph{determined by its moments} if there is no other law besides the law of $X$,  $\mathcal{L}(X)$, having the same moment sequence. The topic of moment determinacy is widely studied, see e.g. \cite{gut2002moment, stoyanov2000krein, lin2009logarithmic, lin2002moment, stoyanov2013hardy2}. A defining topic of previous research has been the discovering of sufficient or necessary conditions for moment determinacy, see \cite{berg1995indeterminate, pakes2007structure, stoyanov2004stieltjes}. 

As a theoretical consequence of the proof technique developed for Theorem \ref{theorem1}, we derive sufficient conditions for the moment determinacy of $S_N$. These conditions can be verified using the tail functions of $X$ and $N$. The moment determinacy of random sums has been previously studied in e.g. \cite{lin2002moment}. So far, it seems that the stopping variable $N$ has been restricted to a narrow class of light-tailed distributions. This is why the present paper offers conditions that apply for a wide class of heavy-tailed distributions and require fewer assumptions than previous results.

\subsection{Structure of the paper}

Basic preliminary information is recalled and defined in Section \ref{prelims}. All of the results are presented in Section \ref{results}. For the convenience of the reader, explanations and applications illustrating the results are presented directly after the exposition of the results. The proofs of results are postponed to Section \ref{proofs}. Technical details omitted during the main text are presented as appendices in Section \ref{appendixsec}.
 
\subsection{Preliminaries and definitions}\label{prelims}

In \cite{oma2}, it is established that the heaviness of a random variable can be measured using a \emph{natural scale}. This is recalled in Lemma \ref{kertauslemma} below. A natural scale of the random variable $X$ approximates the growth speed of the \emph{hazard function} 
$$R_X(x):=-\log P(X>x), $$
where $:=$ denotes equality by definition. The key observation is that the scale function can be chosen to be concave for heavy-tailed random variables. Basic properties of concave functions are tacitly assumed to be known, but can be recalled from e.g. \cite{rockafellar1997convex}.

\begin{definition} A random variable $X$ is (right) \emph{heavy-tailed}, if $E(e^{sX})=\infty$ for all $s>0$. A random variable that is not heavy-tailed, is \emph{light-tailed}.
\end{definition}

If the random variable related to function $h$ or $R$ is clear from the context, the lower index is omitted. The same holds for a general member of an IID sequence. Most results are formulated for the right tails of random variables. To study left tails one can replace $X$ by $-X$. Lastly, for a real number $x$, set $x^+:=\max(0,x)$ and $x^-=\max(0,-x)$.

\begin{lemma}(Recalled from \cite{oma2})\label{kertauslemma} Suppose $X$ is non-negative and heavy-tailed.
\begin{enumerate}
\item \label{ehto1}There exists a concave function $h_X\colon [0,\infty )\to [0,\infty )$ such that $h_X(0)=0$, $h_X(x) \to \infty$, as $x \to \infty$ and
\begin{equation}\label{mimaar}
\liminf_{x \to \infty} \frac{ R(x)}{h_X(x)}=1.
\end{equation}
Such a function $h_X$ is called a natural scale of the random variable $X$.
\item \label{ehto2} For any continuous and increasing function $h \colon [0,\infty) \to [0,\infty)$ such that $h(x)\to \infty$, as $x \to \infty$, we have 
\begin{equation}\label{yhtmaar}
\I_h(X):=\liminf_{x \to \infty} \frac{ R(x)}{h(x)}=\sup\{s\geq 0: E(e^{sh(X)})<\infty\}.
\end{equation}
\item \label{ehto3}  For a natural scale $h_X$ it holds that
\begin{center}
$ \left\{ \begin{array}{ll}
        E(e^{sh_X(X)})<\infty, & \mbox{if $s<1$}\\
        E(e^{sh_X(X)})=\infty, & \mbox{if $s>1$}.\end{array} \right. $ 
\end{center}
\item \label{ehto4}Suppose concave functions $f_1,f_2\colon [0,\infty) \to [0,\infty)$ with $f_1(0)=f_2(0)=0$ and
\begin{center}
$ \left\{ \begin{array}{ll}
        E(e^{f_1(X)})=\infty \mbox{ and}\\
         E(e^{f_2(X)})<\infty.\end{array} \right. $ 
\end{center}
are given. Then, it is possible to choose natural concave scales $h_X$ and $h_X^*$ so that $f_1\geq h_X$ and $h_X^* \geq f_2$ while conditions of  \ref{ehto1} are satisfied.
\end{enumerate}
\end{lemma}

\begin{definition}
Let $X$ be a heavy-tailed random variable. Define
\begin{equation}\label{classmaar}
\mathcal{I}(X):=\{h_X:\mbox{$h_X$ is a natural scale of $X$ with $h_X(0)=0$ }\}.
\end{equation}
\end{definition}
The set $\mathcal{I}(X)$ is non-empty for any heavy-tailed random variable by part \ref{kohta1} of Lemma \ref{kertauslemma}. In addition,  the following properties hold.
\begin{lemma}\label{ekalemma} The following properties hold for natural scales. 
\begin{enumerate} 
\item \label{l2part1}
If $h \in \mathcal{I}(X)$ and $g\sim h$, then 
$$\liminf_{x \to \infty} \frac{R(X)}{g(x)}=1. $$
\item \label{l2part2}
If $h,r \in \mathcal{I}(X)$, then
$$\liminf_{x \to \infty} \frac{h(x)}{r(x)}\leq 1 \mbox{ and } \liminf_{x \to \infty} \frac{r(x)}{h(x)}\leq 1.$$
\item \label{l2part3}
If $c>0$ and $h_X\in \mathcal{I}(X)$, then $g\in \mathcal{I}(cX)$, where $g(x)=h_X(cx)$.
\end{enumerate}
\begin{proof}Clear from definitions and Lemma \ref{kertauslemma}.
\end{proof}
\end{lemma}

\begin{remark}\label{valiremark} Is is possible to find a single function $h_X\in \mathcal{I}(X)$ that satisfies $f_1\geq h_X \geq f_2$ in Lemma \ref{kertauslemma} part \ref{ehto4}. Details can be found from Appendix \ref{remark1app} below.
\end{remark}

\section{Results}\label{results}

We make the following assumptions throughout the rest of the Section \ref{results}:
\begin{enumerate}[I]
\item The sequence of increments $(X_i)$ is an IID-sequence\label{ol1}
\item The stopping variable $N$ is independent of $(X_i)$.\label{ol2}
\end{enumerate}
Under assumptions \ref{ol1}-\ref{ol2} the variable $S_N$ of Equation \eqref{perus} is an independently stopped random walk.

\subsection{Moments of randomly stopped sums}\label{moments}
The \emph{moment index} of a random variable $X$ is defined by the formula
\begin{equation}\label{momenttiindeksi}
\I(X):=\sup\{s\geq 0: E((X^+)^s)<\infty\}\in [0,\infty].
\end{equation}
A small moment index signifies high risk in the sense of a heavy (right) tail. In Proposition \ref{momentprop}, the moment index of $S_N$ is sought. Although parts of the proposition are already know from \cite{gut2009stopped,gut1986converse}, Proposition \ref{momentprop} alleviates requirements on the integrability further. In addition, Formula \eqref{momtulos2} provides both upper and lower bounds for the case of unbounded expectations.

\begin{proposition}\label{momentprop} Suppose assumptions \ref{ol1}-\ref{ol2} are valid. Assume $S_n \to \infty$ almost surely. If at least one of the expectations $E(|X|)$ or $E(N)$ is finite, then
\begin{equation}\label{momtulos}
\I(S_N)=\min(\I(X),\I(N)).
\end{equation}
If $E(|X|)=E(N)=\infty$, it holds that
\begin{equation}\label{momtulos2}
\I(X) \I(N)\leq \I(S_N)\leq \min(\I(X),\I(N)). 
\end{equation}
\end{proposition}

\begin{remark} Condition $S_n\to \infty$ is valid when $E(X)>0$. If $X$ is not integrable, the situation is more delicate. This is discussed in Section \ref{positivedrift}. 
\end{remark}

In the general case, where no assumption is made about the drift of the underlying random walk, Proposition \ref{momtulos} enables the following corollary.

\begin{corollary}\label{momcor} Suppose $E(|X|)<\infty$ and $\I(X)\leq \I(N)$. Then 
\begin{equation}
\I(S_N)=\I(X).
\end{equation}
\end{corollary}

The interpretation of Corollary \ref{momcor} is immediate: At the level of moments, the stopped random walk can exhibit exotic behaviour only when the stopping variable has heavier tail than the increment.

\subsection{Scales of randomly stopped sums}\label{scales}

Here, the results of Section \ref{moments} are further generalised.

\begin{theorem}\label{theorem1}Suppose assumptions \ref{ol1}-\ref{ol2} are valid. Let $X$ and $N$ be heavy-tailed random variables. Assume $0<E(X)<\infty$.
\begin{enumerate}
\item \label{kohta1} Suppose $E((X^+)^{1+\delta})<\infty$ for some $\delta>0$. Let $h_X$ be a natural scale of $X$ such that $h_X(x)\geq (1+\delta)\log x$ for all large enough $x$. If
\begin{equation}\label{th1eq1}
 \liminf_{x \to \infty} \frac{h_{N}(c_1 x)}{h_{X}(x)}\in [1,\infty]
\end{equation}
holds for some natural scale $h_N$ of $N$ and some $c_1>E(X)$, then 
$$h_X \in \mathcal{I}(S_N).$$

\item \label{kohta2} Suppose $E(N^{1+\delta})<\infty$ for some $\delta>0$. Let $h_N$ be a natural scale of $N$ such that $h_N(x)\geq (1+\delta)\log x$ for all large enough $x$. If
\begin{equation}\label{th1eq2}
\liminf_{x \to \infty} \frac{h_X(x)}{h_{ N}(E(X)x)}\in [1,\infty] ,
\end{equation}
holds for some natural scale $h_X$ of $X$ then
$$h_{E(X) N}\in \mathcal{I}(S_N). $$
\end{enumerate}
\end{theorem}

\begin{corollary}\label{corollary1} Suppose $X\geq 0$ almost surely. Assume that $h_X$ satisfies conditions of Theorem \ref{theorem1} part \ref{kohta1}. If it also holds that
\begin{equation}\label{cor1eq1}
\lim_{x \to \infty} \frac{ R_X(x)}{h_X(x)}=1,
\end{equation}
then \eqref{th1eq1} implies $R_{S_N}(x)\sim h_X(x)$, as $x \to \infty$.

Similarly, if $h_X$ satisfies conditions of Theorem \ref{theorem1} part \ref{kohta2} and 
\begin{equation}\label{cor1eq2}
\lim_{x \to \infty} \frac{ R_{E(x)N}(x)}{h_{E(X)N}(x)}=1,
\end{equation}
then \eqref{th1eq2} implies $R_{S_N}(x)\sim h_{E(X)N}(x)$, as $x \to \infty$.
\end{corollary}

\begin{remark}\label{remarkkk}
In Theorem \ref{theorem1} both of the variables $X$ and $N$ are assumed to be heavy-tailed. However, as mentioned in the introduction, cases where heavy-tailed increment is put against a light-tailed stopping or vice versa are straightforward. See Appendix \ref{lhappendix} for details.
\end{remark}

\begin{remark} In parts \ref{kohta1} and \ref{kohta2} of Theorem \ref{theorem1} the functions $h_X$ and $h_{E(X)N}$ that satisfy $h_X(x)\geq (1+\delta)\log x$ and $h_{E(X)N}(x)\geq (1+\delta)\log x$ for all $x\geq 0$ can be found from e.g Lemma \ref{kertauslemma}, part \ref{ehto4}. Once the functions are found, only conditions \eqref{th1eq1} and \eqref{th1eq2} remain to be verified.
\end{remark}

\begin{remark}\label{dominant} If $R_X \in \mathcal{I}(X)$ and $R_{E(X)N} \in \mathcal{I}(E(X)N)$ for some heavy-tailed variables $X$ and $E(X)N$, then the hazard functions themselves can be used as scales. Furthermore, 
$$P(E(X)N>x)=o(P(X>x)) \mbox{, as } x\to \infty$$ implies \eqref{th1eq1} while
$$P(X>x)=o(P(E(X)N>x))  \mbox{, as } x\to \infty$$ implies \eqref{th1eq2}. 
\end{remark}

\begin{remark}\label{crossremark}
The use of scales enables considerations where the difference in behaviours of tails of $X$ and $N$ is more delicate than asymptotic dominance as in Remark \ref{dominant}. In fact, the tail functions of $X$ and $N$ may even cross infinitely many times as long as the associated scales satisfy \eqref{th1eq1} or \eqref{th1eq2}. This phenomenon is illustrated in Example \ref{crossexample} of Section \ref{applications}.
\end{remark}

\subsection{Moment determinacy of random sums}

As recalled in the introduction, moment determinacy (in Stieltjes' sense) of random variable $X$ means that for all non-negative random variables $Y$:
\begin{equation}\label{momimp}
E(X^k)=E(Y^k), \, \forall k \in \mathbb{N} \Longrightarrow X\eqdist Y.
\end{equation}
If implication \eqref{momimp} does not hold, the distribution is called \emph{moment indeterminate}. In this Section we study the moment determinacy of random sums. More precisely, we wish to find sufficient conditions ensuring that the compounding operation preserves moment determinacy.

We begin with a lemma that reformulates the Hardy's condition presented in \cite{stoyanov2013hardy2} via tail functions. The condition gives an asymptotic test for the moment determinacy involving the concept of natural scale.

\begin{lemma}\label{momlemma}
Suppose $X\geq 0$ is a random variable. If
\begin{equation}\label{momdet2}
\liminf_{x \to \infty} \frac{R_X(x)}{\sqrt{x}}\in (0,\infty],
\end{equation}
then $X$ is determined by its moments. In addition, if a natural scale  $h_X$ of a heavy-tailed $X$ satisfies
\begin{equation}\label{momdet}
\liminf_{x \to \infty} \frac{h_X(x)}{\sqrt{x}}\in (0,\infty],
\end{equation}
then \eqref{momdet2} holds and thus $X$ is determined by its moments.
\end{lemma} 

\begin{remark}
In Lemma \ref{momlemma}, the decision of moment determinacy is done using the asymptotic properties of the tail function, which makes small values of $X$ irrelevant. In particular, no assumption about the absolute continuity w.r.t. the Lebesgue measure is needed. These two properties make the test easily applicable compared to other tests such as Carleman condition or finiteness of the Krein integral combined with the Lin condition. For these tests, see \cite{gut2002moment, stoyanov2000krein}.
\end{remark}

Lemma \ref{momlemma} and techniques used in the proof of Theorem \ref{theorem1} enable the following result. 

\begin{theorem}\label{theorem22} Suppose assumptions \ref{ol1}-\ref{ol2} are valid.  Assume that $X$ and $N$ are heavy-tailed and $X$ is non-negative. Assume further that one of the conditions \ref{ehto21} or \ref{ehto22} holds:
\begin{enumerate}
\item \label{ehto21} There exists a natural scale $h_X\in \mathcal{I}(X)$ such that 
\begin{equation}\label{assumption1}
 \liminf_{x \to \infty} \frac{h_X(x)}{\sqrt{x}}\in (0,\infty]
\end{equation}
and 
\begin{equation}\label{th2eq1}
 \liminf_{x \to \infty} \frac{h_{N}(c_1 x)}{h_{X}(x)}\in (0,\infty]
\end{equation}
for some scale $h_N\in \mathcal{I}(N)$ and some $c_1>E(X)$.
\item \label{ehto22} There exists a natural scale $h_N \in \mathcal{I}(N)$  such that 
\begin{equation}\label{assumption2}
\liminf_{x \to \infty} \frac{h_{N}(E(X)x)}{\sqrt{x}}\in (0,\infty]
\end{equation}
and 
\begin{equation}\label{th2eq2}
\liminf_{x \to \infty} \frac{h_X(x)}{h_{ N}(E(X)x)}\in (0,\infty] ,
\end{equation}
for some scale $h_X\in \mathcal{I}(X)$
\end{enumerate}

Then $S_N$ is determined by its moments.
\end{theorem}

\section{Explanations with applications}\label{applications}

In the following Section we illustrate why some of the assumptions are made.

\subsection{The necessity of the positive drift}\label{positivedrift}

In Theorems \ref{theorem1} and \ref{theorem22} a positive drift of the underlying random walk is assumed. If this assumption was changed, the asymptotics of the stopped random walk $S_N$ would change significantly. 

In general, a random walk $(S_n)$ exhibits almost surely one of the following behaviours:
\begin{enumerate}
\item $S_n \to \infty$, as $n \to \infty$ (positive drift) \label{s1}
\item $S_n \to -\infty$, as $n \to \infty$ (negative drift) \label{s2}
\item $\liminf_{n \to \infty}S_n=-\infty$ and $\limsup_{x \to \infty}S_n=\infty$. \label{s3}
\end{enumerate}
If the increment $X$ satisfies $E(|X|)<\infty$, it is know that conditions \ref{s1}, \ref{s2} and \ref{s3} are equivalent with conditions $E(X)>0$, $E(X)<0$ and $E(X)=0$, respectively. For this and further characterisations of the conditions \ref{s1}-\ref{s3}, see \cite{erickson2005} or the beginning of \cite{kesten1996two}.

The following example demonstrates what can happen at the level of moments if the increment has a negative expectation.

\begin{example}\label{hyvaesim} Assume $(S_n)$ is a random walk, $S_n=X_1+\ldots + X_n$, and $N$ is independent of $(X_i)$. Assume further that
\begin{equation}\label{oletusyht}
P(X>x)\sim x^{-\alpha} ,
\end{equation}
as $x \to \infty$, where $\alpha\in (1,\infty)$. Suppose $X\geq -M$ for some number $M>0$ and $E(X)<0$. Define the point density function of $N$ by
$$P(N=k)=\frac{C}{k^{1+\eta}}, $$
where $k \in \mathbb{N}$, $C$ is a norming constant and $\eta>0$. With these choices the moment indices, defined in \eqref{momenttiindeksi}, satisfy $\I(X)=\alpha$ and $\I(N)=\eta$.

It can be shown using assumption \eqref{oletusyht} that for any $a>0$
\begin{equation}\label{kokk}
\liminf_{n \to \infty} \frac{ \log P(S_n>na)}{\log n}\geq 1-\alpha
\end{equation}
holds. See remark \ref{viimremark} for details.

For a small $\epsilon>0$, equation \eqref{kokk} ensures that there is a number $n_\epsilon \in \mathbb{N}$, such that
$$P(S_n>na)\geq n^{1-\alpha-\epsilon}, $$
when $n\geq n_\epsilon$. Hence, for a fixed $a>0$ we obtain
\begin{eqnarray}
E(((S_N)^+)^s)&\geq & a^s \sum_{k=1}^\infty k^s P(S_k>ka)P(N=k) \nonumber \\
&\geq &C a^s \sum_{k=n_\epsilon}^\infty k^s k^{1-\alpha-\epsilon}k^{-(1+\eta)} \nonumber \\
&=&C a^s \sum_{k=n_\epsilon}^\infty k^{s-\alpha-\eta-\epsilon}.\label{vikasumma}
\end{eqnarray}
The series in \eqref{vikasumma} diverges exactly when $s\geq \alpha+\eta+\epsilon-1$. This means that $\I(S_N)\leq \alpha+\eta+\epsilon-1$. Letting $\epsilon \to 0$ yields
\begin{equation}\label{ylayht}
\I(S_N)\leq \alpha+\eta-1.
\end{equation}
In addition to formula \eqref{ylayht}, a lower bound for the quantity $\I(S_N)$ can be obtained from the fact 
\begin{equation}\label{suparvo1}
S_N\leq \bar{S},
\end{equation}
where 
$$\bar{S}=\sup_{k \in \mathbb{N}} S_k. $$ 
Estimate \eqref{suparvo1} implies $\I(\bar{S})\leq \I(S_N)$. On the other hand, the moment index $\I(\bar{S})$ of the supremum of the random walk $(S_n)$ satisfies $\I(\bar{S})=\I(X)-1$. This follows from e.g. Theorem 5.2 of \cite{foss2011introduction} or from \cite{borovkov1976stochastic}, page 140 formula $(54)$. In summary, we get
\begin{equation}\label{ylayht2}
\alpha-1 \leq \I(S_N)\leq \alpha+\eta-1. 
\end{equation}

If $\eta<1$ and, say, $\alpha>2$, the bounds of Equation \eqref{ylayht2} show that the asymptotic behaviour of $S_N$ does not correspond to Formula \eqref{momtulos}. The discrepancy appears because the positive drift is replaced by a negative drift.
\end{example}

Example \ref{tokaex} demonstrates that the behaviour of $S_N$ can also be different from \eqref{momtulos}, when $E(X)=0$.

\begin{example}\label{tokaex} If $E(X)=0$, the Marcinkiewicz-Zygmund inequalities and their generalisations become available. These are used in Section 1.5 of \cite{gut2009stopped}. In fact, in Part (iii) of Theorem 5.1. of \cite{gut2009stopped} the following result is shown: If $E(|X|)<\infty$ and $E(X)=0$ it holds for $s\geq 2$ that
\begin{equation}\label{MZineq}
E(|S_N|^s) \leq C_s E(|X|^s)E(N^{s/2}),
\end{equation}
where $C_s$ is a positive constant.
Formula \eqref{MZineq} implies
$$\I(S_N)\geq \min(\I(|X|),2\I(N))=\min(\I(X^+),\I(X^-),2\I(N)),$$
when $\min(\I(X^+),\I(X^-))\geq  2$ and $\I(N)\geq 1$. Clearly, this is not consistent with Equation \eqref{momtulos}.
\end{example}

The next example demonstrates the possibility that the stopping $N$ can determine the behaviour of $S_N$ even if the increment has a symmetric distribution. 
\begin{example}\label{kolmesim} Suppose that the increment $X$ is Cauchy distributed with density
$$f(x)=\frac{1}{\pi}\frac{1}{(1+x)^2}, \quad x \in \mathbb{R}. $$
Then, it is well known that for any $n \in \mathbb{N}$:
\begin{equation}\label{cauchy}
\frac{S_n}{n}\stackrel{d}{=}X
\end{equation}
holds. Using \eqref{cauchy} it is clear that
$$E(((S_N)^+)^s)= \sum_{k=1}^\infty k^sE(((S_k/k)^+)^s)P(N=k)=E((X^+)^s)E(N^s), $$
which implies 
$$ \I(S_N)=\min(\I(N),\I(X)).$$ 
This is the content of \eqref{momtulos}. 

\end{example}
Examples \ref{hyvaesim}, \ref{tokaex} and \ref{kolmesim} together show that the behaviour of moments of $S_N$ is universally determined by \eqref{momtulos} only when the drift is positive. Without the positive drift behaviour of $S_N$ can change or stay the same.

\subsection{Zigzagging tail functions}

As mentioned in Remark \ref{crossremark}, scales of hazard functions may be used in situations where the tails cross infinitely often. The following example illustrates how scales can still be obtained and thus Theorem \ref{theorem1} applied.

\begin{example}\label{crossexample} 
We return to the setting of Appendix A.1 of \cite{oma2}.  Suppose we are given functions 
$$h_i \colon [0,\infty) \to [0,\infty),$$ 
where $i=1,2,3,4$ that satisfy
\begin{enumerate}[I]
\item $h_i(0)=0$ and  $h_i$ is strictly increasing and continuous for every $i$
\item $h_j(x)<h_i(x)$ when $j<i$ for all $x\geq 0$
\item $h_1$ and $h_2$ are concave.
\end{enumerate} 

\begin{figure}
 \centering
 \def\svgwidth{\columnwidth}
 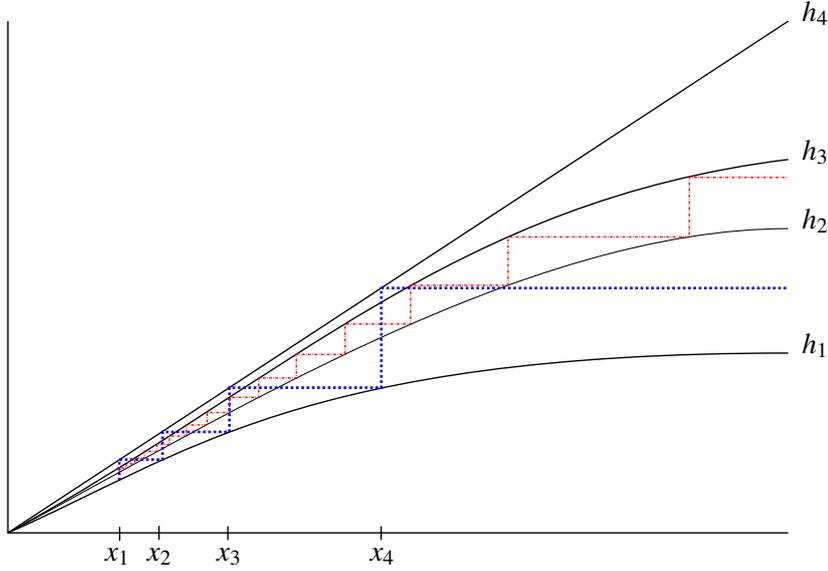
 \caption{Illustration of the construction of functions $R_X$ (blue) and $R_Y$ (red).  \label{fig2} }
\end{figure}
First, we define the sequences $(x_n)$ and $(y_n)$ are defined by requirements
\begin{enumerate}
\item $x_0=y_0=1$ 
\item $x_{n+1}=h_1^{-1}(h_4(x_n))$
\item $y_{n+1}=h_2^{-1}(h_3(x_n))$,
\end{enumerate}
where $n \in \{0,1,2,\ldots\}$. Then, we construct random variables $X$ and $Y$, which are concentrated to the sets $\{x_n: n\in \mathbb{N}\}$ and $\{y_n: n\in \mathbb{N}\}$, respectively. We set 
$$P(X=x_n)= e^{-h_4(x_n)}-e^{-h_1(x_n)}$$
and
$$P(Y=y_n)= e^{-h_3(y_n)}-e^{-h_2(y_n)}$$
for $n \in \{0,1,2,\ldots\}$. The procedure is illustrated in Figure \ref{fig2}.

Now, by construction, it must hold that
\begin{equation}\label{esimy1}
\limsup_{x \to \infty} \frac{R_X(x)}{h_4(x)}=\limsup_{x \to \infty} \frac{R_Y(x)}{h_3(x)}=\liminf_{x \to \infty} \frac{R_Y(x)}{h_2(x)}=\liminf_{x \to \infty} \frac{R_X(x)}{h_1(x)}=1.
\end{equation}
By continuity of the functions $h_i$ the difference of hazard functions $R_X(x)-R_Y(x)$ must change signs infinitely many times. This implies that the tails must also cross infinitely many times. However, from \eqref{esimy1} we may still conclude that $h_1\in \mathcal{I}(X)$ and $h_2\in \mathcal{I}(Y).$
 
\end{example}

\section{Proofs}\label{proofs}

\begin{proof}(Proof of Proposition \ref{momentprop})
We divide the proof into separate cases.
\begin{enumerate}
\item We will first prove \eqref{momtulos}. To see that 
$$\I(S_N) \geq \min(\I(N),\I(X)),$$ 
we can use one of the following tools:
\begin{enumerate}
\item \label{aa} Subadditivity of the function $x\mapsto x^s$ for $x\geq 0$, 
\item \label{bb} Jensen's inequality or
\item \label{cc} Minkowski's inequality.
\end{enumerate}

Depending on the value of $s$, the aim is to bound the expectation $E((S_N^+)^s)$ from above. Suppose $0<s\leq 1$. Because of the fact $(S_N)^+\leq \sum_{k=1}^N(X_k)^+$ and the subadditivity it must hold that
\begin{eqnarray}
E((S_N^+)^s)&\leq & \sum_{k=1}^\infty E \left( \left( \sum_{m=1}^k X_m^+\right)^s \right)P(N=k) \nonumber\\
&\leq & \sum_{k=1}^\infty kE ( (X^+)^s)P(N=k)\nonumber  \\
&=& E ( (X^+)^s)E(N).\label{sub}
\end{eqnarray}
On the other hand, applying the Jensen's inequality to the concave function $x\mapsto x^s$ we get
\begin{eqnarray}
E((S_N^+)^s)&\leq & \sum_{k=1}^\infty E \left( \left( \sum_{m=1}^k X_m^+\right)^s \right)P(N=k) \nonumber \\
&\leq & \sum_{k=1}^\infty k^sE ( X^+)^sP(N=k) \nonumber \\
&=& (E ( X^+))^sE(N^s). \label{jens}
\end{eqnarray}

Suppose then $s>1$. By Minkowski's inequality we have 
\begin{eqnarray}
E((S_N^+)^s)&\leq & \sum_{k=1}^\infty E \left( \left( \sum_{m=1}^k X_m^+\right)^s \right)P(N=k) \nonumber\\
&\leq & \sum_{k=1}^\infty k^sE ( (X^+)^s)P(N=k)\nonumber  \\
&=& E ( (X^+)^s)E(N^s).\label{mink}
\end{eqnarray}
Equations \eqref{sub}, \eqref{jens} and \eqref{mink} imply $\I(S_N) \geq \min(\I(N),\I(X))$ under the assumption that one of the expectations $E(X^+)$ or $E(N)$ is finite. 

We will then prove $\I(S_N) \leq \min(\I(N),\I(X))$. Firstly, because 
\begin{equation}\label{trivial}
(S_N)^+\geq X^+ \mathbf{1}(N=1),
\end{equation}
we have that $\I(S_N) \leq \I(X)$. In the case of a finite expectation we note that for any $a\in (0,E(X))$ it holds for a positive constant $C$ that $P(S_k>ak)\geq C$ for some $C>0$ that does not depend on $k$. This is clear from the law of large numbers. In the remaining case where $E(|X|)=\infty$ but $S_n\to \infty$ almost surely we have, see the beginning of e.g. \cite{kesten1996two}, that $S_n/n\to \infty$ almost surely. Hence
\begin{eqnarray*}
E((S_N^+)^s)& \geq &  \sum_{k=1}^\infty E \left( \left( S_k^+ \right)^s \right)P(N=k) \nonumber \\
&\geq &  \sum_{k=1}^\infty E \left( \left( S_k^+\right)^s \mathbf{1}(S_k>ak)\right)P(N=k) \nonumber \\
& \geq & a^s\sum_{k=1}^\infty k^s P(S_k>ak)P(N=k) \nonumber \\
& \geq & a^sCE(N^s),
\end{eqnarray*}
which proves $\I(S_N) \leq \I(N)$. This ends the proof of \eqref{momtulos}.

\item We are set to prove \eqref{momtulos2}. Since the proof of the inequality $\I(S_N) \leq \min(\I(N),\I(X))$ did not require any assumptions about the finiteness of the expected values $E(|X|)$ and $E(N)$, we only need to show $\I(S_N) \geq \I(X) \I(N).$ Without any loss of generality we can assume that $\min(\I(N),\I(X))>0$.

Fix a small number $\epsilon>0$. Set $q_1:=\I(X)-\epsilon$, $q_2:=\I(N)$ and $q:=q_1 q_2$. Now, for $0<s<1$, it must hold by subadditivity and Jensen's inequality that
\begin{eqnarray}
E(((S_N^+)^q)^s)&\leq & \sum_{k=1}^\infty E \left( \left( \left(\sum_{m=1}^k X_m^+\right)^{q_1} \right)^{q_2s} \right)P(N=k) \nonumber\\
&\leq & \sum_{k=1}^\infty E \left( \left( \sum_{m=1}^k \left( X_m^+\right)^{q_1} \right)^{q_2s} \right)P(N=k) \nonumber \\
&\leq & \sum_{k=1}^\infty k^{q_2s}E ( (X^+)^{q_1})^{q_2s}P(N=k)\nonumber  \\
&=& E ( (X^+)^{q_1})^{q_2s}E(N^{q_2s})<\infty.\label{mink2}
\end{eqnarray}
Equation \eqref{mink2} implies $\I((S_N^+)^q)\geq 1$, which is equivalent to 
$$\I(S_N)\geq q=(\I(X)-\epsilon)\I(N).$$ 
Now, letting $\epsilon \to 0$ proves the claim.

\end{enumerate}
\end{proof}

\begin{proof}(Proof of Corollary \ref{momcor}) Set 
$$ S^a_N:=\sum_{k=1}^N \left( X-E(X)+|E(X)|+1\right) $$
and 
$$ S^b_N:=\sum_{k=1}^N \left( X-E(X)-(|E(X)|+1)\right). $$

Since $ S^{b}_N\leq S_N\leq S^a_N$, it must hold that
\begin{equation}\label{cor1per1}
\I(S^{b}_N) \geq  \I(S_N) \geq \I(S^a_N).
\end{equation}
The underlying random walk of $S^a_N$ is positive, while that of $S^b_N$ is negative.

By Proposition \ref{momentprop}, $\I(S^a_N)=\min(\I(X-E(X)+|E(X)|+1),\I(N))$. Using the basic properties of moment indices, see e.g. Section 2.2 of \cite{oma1}, and the assumption $\I(N)\geq \I(X)$ we have that $\I(S^a_N)=\I(X)$. In addition, similarly as in Equation \eqref{trivial}, we obtain $\I(S^{b}_N)\leq \I(X-E(X)-(|E(X)|+1)=\I(X)$. Equation \eqref{cor1per1} proves the claim.
\end{proof}

The proof of Proposition \ref{momentprop} was based on application of classical inequalities. Once we move on to the study of general scales, the following two technical lemmas become necessary. 

\begin{lemma}\label{ekapau} Suppose that $h \colon [0,\infty) \to [0,\infty)$ is an increasing concave function with $h(0)=0$. Then for all $x\geq 0$:
\begin{center}
$ \left\{ \begin{array}{ll}
        h(cx)\geq ch(x), &  \mbox{if } 0<c<1 \mbox{ and}\\
        h(cx)\leq ch(x), &  \mbox{if } c>1.\end{array} \right. $ 
\end{center}
\begin{proof} Assume $x>0$. Suppose first that $0<c<1$. Using the definition of concavity we have 
$$h(cx)=h(cx+(1-c)0)\geq ch(x)+(1-c)h(0)=ch(x), $$
which proves the first part.

Suppose then that $c>1$. Set $y:=c x$. Applying the first part at point $y$ we get
$$h((1/c)y)\geq (1/c)h(y), $$
that is,
$$h(cx)\leq c h(x) $$
holds and we are done.
\end{proof}
\end{lemma}

\begin{lemma}\label{jvalemma} Let $X$ be a non-negative heavy-tailed random variable and $h_X \in \mathcal{I}(X)$. Then the mapping
$$c \mapsto \I_{h_X}(cX) $$
is continuous at point $c=1$.
\begin{proof} By concavity and the property $h_X(0)=0$, Lemma \ref{ekapau} is applicable.

For a fixed $0<c_1<1$ we may estimate $E(e^{sh_X(c_1X)})\geq E(e^{c_1 s h_X(X)})$ and obtain
\begin{equation}\label{apu1}
\sup \{s\geq 0: E(e^{sh_X(c_1X)})<\infty\} \leq \sup \{s\geq 0: E(e^{sc_1h_X(X)})<\infty\} =\frac{1}{c_1}. 
\end{equation}
Similarly, for a fixed $c_2>1$ we have $E(e^{sh_X(c_2X)})\leq E(e^{c_2 s h_X(X)})$ and so
\begin{equation}\label{apu2}
\sup \{s\geq 0: E(e^{sh_X(c_2X)})<\infty\} \geq \sup \{s\geq 0: E(e^{sc_2h_X(X)})<\infty\} =\frac{1}{c_2}. 
\end{equation}
Letting $c_1\uparrow 1$ in \eqref{apu1} and $c_2 \downarrow 1$ in \eqref{apu2} we finally get that 
$$\lim_{c \to  1} \left( \sup \{s\geq 0: E(e^{sh_X(cX)})<\infty\}\right)= 1, $$
which proves the claim.
\end{proof}
\end{lemma}

Before the proof of Theorem \ref{theorem1} we recall an important lemma concerning the growth rate of transformations of sums of heavy-tailed random variables.

\begin{lemma}(Lemma 3 of \cite{denisov2008lower2})\label{dlemma}
\\
Let $\xi$ be a nonnegative random variable. Let $h \colon [0,\infty) \to [0,\infty)$ be a nondecreasing and eventually concave function such that $h(x)=o(x)$ as $x \to \infty$ and $h(x)\geq \log x$ for all sufficiently large $x$. If $E(e^{h(\xi)})<\infty$, then, for any $c>E(\xi)$, there exists a constant $K(c)$ such that $E(e^{h(S_n)})\leq K(c) e^{h(nc)}$.
\begin{proof} See \cite{denisov2008lower2}, pp. 694-695.
\end{proof}
\end{lemma}

\begin{proof}(Proof of Theorem \ref{theorem1}) 
\begin{enumerate}
\item Note first that Assumption \eqref{th1eq1} implies 
$$\liminf_{x \to \infty} \frac{R_{c_1N}(x)}{h_X(x)} \geq \left( \liminf_{x \to \infty} \frac{R_{c_1N}(x)}{h_N(c_1 x)}\right) \left( \liminf_{x \to \infty} \frac{h_N(c_1 x)}{h_X(x)}\right)\geq 1.$$
Using part \ref{ehto2} of Lemma \ref{kertauslemma} we see that 
\begin{equation}\label{aarpaat}
E(e^{(1-\epsilon)h_X(c_1N)})<\infty
\end{equation}
for any $\epsilon>0$.

Recall from assumptions that $\delta>0$ is fixed and $c_1>E(X)$. Let 
$$0<\eta<\delta/(1+\delta).$$
Using the fact that 
$$(1-\eta)h_X(x)\geq \log x$$
and Lemma \ref{dlemma} we obtain the upper bound
\begin{equation}\label{sovellus1}
E(e^{(1-\eta)h_X(S_n^+)})\leq E(e^{(1-\eta)h_X(X_1^+ +\ldots +X_n^+)})\leq K(c_1)e^{(1-\eta)h_X(c_1 n)}
\end{equation}
for all $n \in \mathbb{N}$, where the constant $K(c_1)$ only depends on $c_1$ but not on $n$.

Utilising estimates \eqref{sovellus1} and \eqref{aarpaat} we get
\begin{eqnarray*}
E(e^{(1-\eta)h_X(S_N^+)})&=&\sum_{k=1}^\infty E(e^{(1-\eta)h_X(S_k^+)})P(N=k) \\
&\leq & \sum_{k=1}^\infty K(c_1)e^{(1-\eta)h_X(c_1 n)} P(N=k) \\
&=& K(c_1)E(e^{(1-\eta)h_X(c_1 N)})<\infty.
\end{eqnarray*}
Since the mapping $s \mapsto E(e^{sh_X(S_N)})$ is increasing,  $E(e^{sh_X(S_N)})<\infty$ holds for all $0<s<1$. Thus $\I_{h_X}(S_N)\geq 1$. On the other hand, for any $\epsilon>0$ it holds that
$$E(e^{(1+\epsilon)h_X(S_N^+)})\geq E(e^{(1+\epsilon)h_X(X^+)})P(N=1)=\infty $$
and so $\I_{h_X}(S_N)\leq 1$. This proves $h_X\in \mathcal{I}(S_N)$.

\item  Suppose $0<\eta_1<\delta/(1+\delta)$ is fixed. From Lemma \ref{jvalemma} we know that the mapping $c\mapsto \I_{h_{E(X)N}}(cN)$ is continuous in a neighbourhood of $c=E(X)$. Here the connection of part \ref{l2part3} of Lemma \ref{ekalemma} is used. Because of continuity we may find a number $c_{\eta_1}>E(X)$ such that 
\begin{equation}\label{arg1}
E(e^{(1-\eta)h_{E(X)N}(c_{\eta_1} N)})<\infty.
\end{equation} 
Using similar arguments as in the first part of the theorem and Equation \eqref{arg1} we get
\begin{eqnarray*}
E(e^{(1-\eta_1)h_{E(X)N}(S_N^+)})&=&\sum_{k=1}^\infty E(e^{(1-\eta_1)h_{E(X) N}(S_k^+)})P(N=k) \label{11}\\
&\leq & \sum_{k=1}^\infty K(c_{\eta_1})e^{(1-\eta_1)h_{E(X) N}(c_{\eta_1} k)} P(N=k) \label{22} \\
&=& K(c_{\eta_1})E(e^{(1-\eta_1)h_{E(X)N}(c_{\eta_1} N)})<\infty. \label{33}
\end{eqnarray*}
This shows $\I_{h_{E(X) N}}(S_N)\geq 1.$

Other direction follows from the law of large numbers: Assume $\eta_2>0$ is given. Then by Lemma \ref{jvalemma} we may choose a number $0<c_{\eta_2}<E(X)$ so that
\begin{equation}\label{arg2}
E(e^{(1+\eta_2)h_{E(X)N}(c_{\eta_2} N)})=\infty.
\end{equation}
In addition, by the law of large numbers we may find a constant $C=C_{\eta_2,c_{\eta_2}}>0$ such that $P(S_k>c_{\eta_2} k)\geq C$ for all $k \in \mathbb{N}$. Now
\begin{eqnarray*}
E(e^{(1+\eta_2)h_{E(X)N}(S_N^+)})&=&\sum_{k=1}^\infty E(e^{(1+\eta_2)h_{E(X) N}(S_k^+)})P(N=k) \\
& \geq & \sum_{k=1}^\infty E(e^{(1+\eta_2)h_{E(X) N}(S_k^+)}\mathbf{1}(S_k>c_{\eta_2} k))P(N=k) \\
&=& \sum_{k=1}^\infty e^{(1+\eta_2)h_{E(X) N}(c_{\eta_2} k)}P(S_k>c_{\eta_2} k)P(N=k) \\
& \geq & C\sum_{k=1}^\infty e^{(1+\eta_2)h_{E(X) N}(c_{\eta_2} k)}P(N=k) \\
&=&C E(e^{(1+\eta_2)h_{E(X)N}(c_{\eta_2} N)})=\infty.
\end{eqnarray*}
This proves $\I_{h_{E(X) N}}(S_N)\leq 1$ and we get $h_{E(X)N}\in \mathcal{I}(S_N)$.
\end{enumerate}
\end{proof}

\begin{proof}(Proof of Corollary \ref{corollary1})
Note first that $P(S_N>x)\geq P(X>x)$ and therefore $R_{S_N}(x)\leq R_X(x)$ holds for all $x$, since $X\geq 0$ almost surely. Application of Theorem \ref{theorem1}, and the assumption \eqref{cor1eq1} gives
$$1=\liminf_{x\to \infty} \frac{R_{S_N}(x)}{h_X(x)}\leq \limsup_{x\to \infty} \frac{R_{S_N}(x)}{h_X(x)}\leq \limsup_{x\to \infty} \frac{R_{X}(x)}{h_X(x)}=1 $$
and proves the claim. The proof of the remaining case is similar.
\end{proof}

\begin{proof}(Proof of Lemma \ref{momlemma})
Implication is easily verified by observing that \eqref{momdet} together with the definition of natural scale and positivity yields
$$\liminf_{x \to \infty} \frac{R(x)}{\sqrt{x}}\geq \left( \liminf_{x \to \infty} \frac{h(x)}{\sqrt{x}} \right) \left(\liminf_{x \to \infty} \frac{R(x)}{h(x)}\right) >0.$$
Hence, by part \ref{ehto2} of Lemma \ref{kertauslemma}, there exists $c>0$ such that 
$$E(e^{c\sqrt{X}})<\infty $$
and Theorem 1 of of \cite{stoyanov2013hardy2} confirms that $X$ is determined by its moments.
\end{proof}

\begin{proof}(Proof of Theorem \ref{theorem22})
The proof utilises similar techniques as the proof of Theorem \ref{theorem1}. 

\begin{enumerate}
\item \label{proofpart1} Suppose first that Condition \ref{ehto21} of Theorem \ref{theorem22} holds. Since
$$\liminf_{x \to \infty} \frac{R_{c_1N}(x)}{h_X(x)} \geq \left( \liminf_{x \to \infty} \frac{R_{c_1N}(x)}{h_N(c_1 x)}\right) \left( \liminf_{x \to \infty} \frac{h_N(c_1 x)}{h_X(x)}\right)>0,$$
there must exist $\eta>0$ so that 
$$E(e^{\eta h_X(c_1N)})<\infty. $$
Clearly, by assumption \eqref{assumption1}, $\eta h_X(x)\geq \log x$ holds for all $x$ large enough. This enables the use of Lemma \ref{dlemma}.

We get 
\begin{eqnarray*}
E(e^{\eta h_X(S_N)})&=&\sum_{k=1}^\infty E(e^{\eta h_X(S_k)})P(N=k) \\
&\leq & \sum_{k=1}^\infty K(c_1)e^{\eta h_X(c_1 n)} P(N=k) \\
&=& K(c_1)E(e^{\eta h_X(c_1 N)})<\infty.
\end{eqnarray*}
This implies, using part of \ref{ehto2} of Lemma \ref{kertauslemma}, that
\begin{equation}\label{th2apu1}
\liminf_{x \to \infty} \frac{R_{S_N}(x)}{h_X(x)}>0.
\end{equation} 
By assumption \eqref{assumption1} there exist a constant $c>0$ such that $h_X(x)\geq c \sqrt{x}$ for all $x$ large enough. Hence \eqref{th2apu1} implies 
$$\liminf_{x \to \infty} \frac{R_{S_N}(x)}{\sqrt{x}}>0 $$
and the claim follows directly from Lemma \ref{momlemma}.

\item The proof of the second part follows same kind of arguments as the first part. First, we see that 
$$\liminf_{x \to \infty} \frac{R_{X}(x)}{h_{E(X)N}(x)} \geq \left( \liminf_{x \to \infty} \frac{R_{X}(x)}{h_X(x)}\right) \left( \liminf_{x \to \infty} \frac{h_X(x)}{h_{ N}(E(X)x)}\right)>0,$$
ensuring the existence of small $\eta>0$ so that 

$$E(e^{\eta h_{E(X)N}(X)})<\infty. $$

Using Lemma \ref{jvalemma} we may find a constant $c_1>E(X)$ so that also 
$$E(e^{\eta h_{E(X)N}(c_1N)})<\infty $$
holds. Now, again, by Lemma \ref{dlemma} we get
\begin{eqnarray*}
E(e^{\eta h_{E(X)N}(S_N)})&=&\sum_{k=1}^\infty E(e^{\eta h_{E(X) N}(S_k)})P(N=k)\\
&\leq & \sum_{k=1}^\infty K(c_\eta)e^{\eta h_{E(X) N}(c_\eta k)} P(N=k)  \\
&=& K(c_1)E(e^{\eta h_{E(X)N}(c_1 N)})<\infty. 
\end{eqnarray*}
The conclusion follows in the same way as in the first part.
\end{enumerate}
\end{proof}

\appendix
\section{Appendices}\label{appendixsec}
Some of the technical details were omitted during the main text. They are presented in this Section.

\subsection{Details for Remark \ref{valiremark}}\label{remark1app}
Let $h_X^*$ be as in part \ref{ehto4} of Lemma \ref{kertauslemma}. If
$$\liminf_{x \to \infty} \frac{h_X^*(x)}{f_1(x)}<1 $$ 
we have that eventually $h_X^*\leq f_1$. In this case setting $h_X=\min(f_1,h_X^*)$ gives the required concave function.

If
$$\liminf_{x \to \infty} \frac{h_X^*(x)}{f_1(x)}\geq 1 $$ 
holds we get using the fact $E(e^{f_1(X)})=\infty$ and part \ref{ehto2} of Lemma \ref{kertauslemma} that
\begin{equation}\label{pitkakaava1}
1\geq \liminf_{x \to \infty} \frac{R_X(x)}{f_1(x)} \geq \left( \liminf_{x \to \infty} \frac{R_X(x)}{h_X^*(x)} \right)\left( \liminf_{x \to \infty} \frac{h_X^*(x)}{f_1(x)} \right)  \geq 1. 
\end{equation}
Formula \eqref{pitkakaava1} implies that $f_1$ must be a natural scale of variable $X$ and thus the function $f_1$ satisfies the requirement of Remark \ref{valiremark}.

\subsection{Light-tailed vs. heavy-tailed: details for Remark \ref{remarkkk}}\label{lhappendix}
We will consider the setting of Theorem \ref{theorem1} with modified assumptions.
\begin{enumerate}
\item Suppose that the variable $N$ is light-tailed and $X$ is heavy-tailed. Then the result of part \ref{kohta1} of Theorem \ref{theorem1} holds when Assumption \eqref{th1eq1} is omitted. The proof of the result simplifies substantially compared to the case where both of the variables are heavy-tailed. 

Firstly, since $N$ is light-tailed, we have that $E(e^{cN})<\infty$ for some $c>0$. Thus, by Lemma \ref{kertauslemma} part \ref{ehto2} and the fact that $h_X(x)=o(x)$:
$$\liminf_{x \to \infty} \frac{R_{c_1N}(x)}{h_X(x)} \geq \left( \liminf_{x \to \infty} \frac{R_{c_1N}(x)}{x}\right) \left( \liminf_{x \to \infty} \frac{x}{h_X(x)}\right)=\infty.$$
This confirms that \eqref{aarpaat} holds and the rest of the proof is similar than in the heavy-tailed case.
\item Suppose that the variable $X$ is light-tailed and $N$ is heavy-tailed. Then a similar modification as in the first part yields the result without Assumption \eqref{th1eq2} of the part \ref{kohta2}.
\end{enumerate}

\subsection{Details for Remark \ref{crossremark}}\label{vigaremark}

We wish to show explicitly that the condition $h_X\geq f$ for some $h_X \in \mathcal{I}(X)$ does not imply that $h_X^*\geq f$ for every $h_X^* \in \mathcal{I}(X)$. Suppose first, for simplicity, that $R_X=R \in \mathcal{I}(X)$. Define $h_X=R$ and let $g$ be any increasing function with $g(x)\to \infty$ as $x \to \infty$ such that $h_X> g$. 

It is now possible to construct the required scale as a piecewise linear function as illustrated in Figure \ref{fig1}. The idea is to iteratively find a line segments whose both endpoints are connected to the graph of $R$, but cross the function $g$. Such segments can always be found since $R(x)=o(x)$, as $x \to \infty$. By selecting the slopes of segments to form a decreasing sequence of numbers ensures that the resulting function is concave. The rest of the required properties are clear from the construction.

\begin{figure}
 \centering
 \def\svgwidth{0.8\columnwidth}
 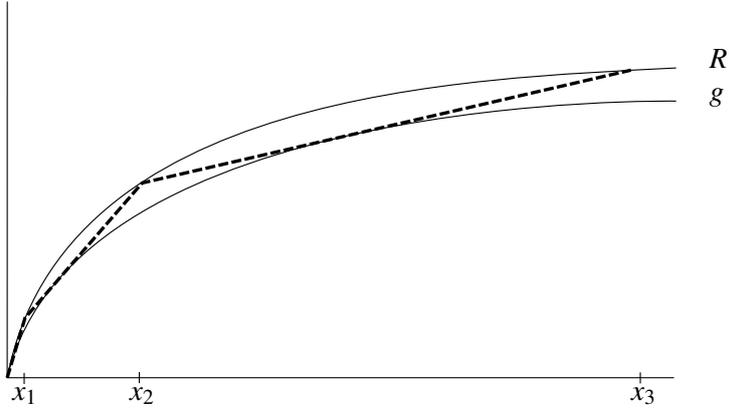
 \caption{The function $h_X$ of Remark \protect\ref{vigaremark} is illustrated in the figure. First three points of the sequence $(x_n)$ are sought. The dashed line is the graph of the function $h_X$ on interval $[0,x_3]$. \label{fig1} }
\end{figure}

\subsection{Details for Example \ref{hyvaesim}}\label{viimremark}
We show that inequality \eqref{kokk} holds. The following estimate is based on a similar deduction presented in the lecture notes \cite{rjatkomon}. Since the notes are in Finnish, the needed argument is fully recalled below. For further analysis of power tailed random walks and their generalisations the reader is advised to see \cite{nyrhinen2005}.

For a fixed $\epsilon>0$ we have for sufficiently large $n$ that
\begin{eqnarray*}
P(S_n>na) &\geq& nP(X_1>n^{1+\epsilon},X_2\leq n^{1+\epsilon},\ldots, X_n\leq n^{1+\epsilon},S_n-X_1>n(E(X)-\epsilon)) \\
&=& n\overline{F}(n^{1+\epsilon})P(X_2\leq n^{1+\epsilon},\ldots, X_n\leq n^{1+\epsilon},S_n-X_1>n(E(X)-\epsilon)) \\
&\sim & n\overline{F}(n^{1+\epsilon}).
\end{eqnarray*}
Hence
\begin{equation}\label{aaaappp}
\liminf_{n \to \infty} \frac{ \log P(S_n>na)}{\log n}\geq 1-(1+\epsilon)\alpha.
\end{equation}
Letting $\epsilon\to 0$ in \eqref{aaaappp} gives \eqref{kokk}.
\section*{Acknowledgements}

The deepest gratitude is expressed to the Finnish Doctoral Programme in Stochastics and Statistics (FDPSS) and the Centre of Excellence in Computational Inference (COIN) for financial support (Academy of Finland grant number 251170). Special thanks are due to Harri Nyrhinen for his diligent guidance throughout the writing of the paper.

\bibliographystyle{acm}
\bibliography{mybib}

\end{document}